\documentclass[11pt]{amsart}
\usepackage[colorlinks]{hyperref}

\providecommand\texorpdfstring[2]{#1}

\headheight\baselineskip
\headsep\baselineskip
\advance\topmargin-\headheight 
\advance\topmargin-\headsep

\renewcommand{\H}{\mathcal{H}}          
\newcommand{\I}{\mathcal{I}}            
\newcommand{\J}{\mathcal{J}}            
\newcommand{\N}{\mathcal{N}}            
\newcommand{\W}{\mathcal{W}}            
\newcommand{\Z}{\mathcal{Z}}            
\newcommand{\ID}{\mathbb{D}}            

\let\a\alpha
\let\g\gamma
\let\b\beta
\let\z\zeta
\newcommand{\del}{\partial}
\newcommand{\dbar}{\bar\partial}

\newcommand{\invL}{\mathop{\tilde\Delta}}

\newcommand{\av}[1]{\left| #1 \right|}
\newcommand{\st}{:}
\newcommand{\re}{\mathop{\mathrm{Re}}\nolimits}
\newcommand{\im}{\mathop{\mathrm{Im}}\nolimits}

\let\eps\epsilon
\let\intersect\cap

\let\union\cup
\let\Union\bigcup
\let\phi\varphi
\let\term\emph

\newtheorem{theorem}{Theorem}[section]
\newtheorem{lemma}[theorem]{Lemma}
\newtheorem{proposition}[theorem]{Proposition}

\theoremstyle{remark}

\numberwithin{equation}{section}

\begin{document}

\title[Interpolation Schemes]%
{Interpolation Schemes in Weighted Bergman Spaces}
\author{Daniel H. Luecking}
\address{Department of Mathematical Sciences\\
         University of Arkansas\\
         Fayetteville, Arkansas 72701}
\email{luecking@uark.edu}
\date{June 6, 2014}

\subjclass{Primary 46E20}
\keywords{Bergman space, interpolating sequence, upper density,
uniformly discrete}

\begin{abstract}
    We extend our development of interpolation schemes in \cite{Lue04b}
    to more general weighted Bergman spaces.
\end{abstract}

\maketitle

\section{Introduction}

Let $A$ denote area measure and let $G$ be a domain in the complex
plain. Let $\H(G)$ denote the space of holomorphic functions on $G$ and
$L^p(G)= L^p(G,dA)$ the usual Lebesgue space of measurable functions
$f$ with $\| f \|_{p,G}^p = \int_G |f|^p \,dA < \infty$. The
\term{Bergman space $A^p(G)$} is $L^p(G) \intersect
\H(G)$, a closed subspace of $L^p(G)$. If $1 \le p < \infty$, $A^p(G)$ are
Banach spaces and if $0<p<1$ they are quasi-Banach spaces. We will allow
all $0<p<\infty$ and abuse the terminology by calling $\| \cdot
\|_{p,G}$ a norm even when $p < 1$.  In the case where $G = \ID$, the
open unit disk, we will abbreviate: $L^p = L^p(\ID)$, $A^p=
A^p(\ID)$~and $\| \cdot \|_p = \| \cdot \|_{p,\ID}$.

Let $\psi(z,\z)$ denote the \term{pseudohyperbolic metric}:
\begin{equation*}
  \psi(z,\z) = \left|\frac{z-\z}{1-\bar\z z}\right|.
\end{equation*}
We will use $D(z,r)$ for the \term{pseudohyperbolic disk} of radius $r$
centered at $z$, that is, the ball of radius $r<1$ in the
pseudohyperbolic metric. Let $d\lambda(z) = (1 - |z|^2)^{-2}dA(z)$
denote the \term{invariant area measure} on $\ID$.

We abbreviate derivatives $\partial/\partial z$ and $\partial/\partial
\bar z$ by $\del$ and $\dbar$ and the combination $\del\dbar u$ will be
called the \term{Laplacian of $u$}. The \term{invariant Laplacian} of $u$, denoted
$\invL u$, is defined by $\invL u(z) = (1 - |z|^2)^2\del\dbar u(z)$.

Let $\phi$ be a $C^2$ function in $\ID$ satisfying $0 < m \le \invL
\phi(z) \le M < \infty$, for positive constants $m$ and $M$. We define the
weighted Bergman space $A_\phi^p$ to consist of all functions $f$ that
are analytic in $\ID$ and satisfy the following
\begin{equation}
  \| f \|_{\phi,p} = \left( \int_\ID  \frac{\av{f(z)e^{-\phi(z)}}^p} {1 -
  |z|^2} \,dA(z)\right)^{1/p} < \infty
\end{equation}
With $p=2$ only, these spaces were considered by A. Schuster and T.
Wertz in \cite{SW13} (our formulation differs by a factor of 2 in
$\phi$). In that paper, a necessary condition was obtained for a certain
weighted interpolation problem they called O-\term{interpolation}
(presumably after its origins in a paper by S. Ostrovsky \cite{Ost10}).

The purpose of this paper is to extend the current author's results in
\cite{Lue04b} to these more general weighted Bergman spaces, and
as a consequence to extend the results of \cite{SW13} to $p \ne 2$.

Following \cite{Lue04b}, we define an \term{interpolation scheme $\I$}
to consist of connected open sets $G_k \subset \ID$, $k=1,2,3,\dots$ and
corresponding disjoint finite nonempty multisets $Z_k \subset G_k$
(\term{multisets} are sets with multiplicity) satisfying the following
\begin{enumerate}
  \item there exists $\eps>0$ such that $(\Z_k)_\eps \subset G_k$ for
        every $k$, and
  \item there exists $0<R<1$ such that for every $k$ the
        pseudohyperbolic diameter of $G_k$ is no more than $R$.
\end{enumerate}
The notation $(S)_\eps$ for a subset $S\subset \ID$ denotes the
$\eps$-neighborhood of $S$ (in the pseudohyperbolic metric), and the
\term{pseudohyperbolic diameter} of a set $S\subset \ID$ is $\sup\{
\psi(z,w)\st z,w\in S \}$. We remark that $G_k$ are not required to be
disjoint. They are also not required to be simply connected, but it is
no real loss of generality to assume that they are, or even to
assume that $G_k$ are pseudohyperbolic disks of constant radius.

Since finite sets are trivial for our problem, we will always assume the
number of clusters is countably infinite.

Given a pair $(G_k, Z_k)$ in an interpolation scheme $\I$, let $\N_k$
consist of all functions in $\H(G_k)$ (holomorphic on $G_k$) that vanish
on $Z_k$ with the given multiplicities. An interpolation problem can be
thought of as specifying values for $f$ and its derivatives at the
points of $\Z = \Union Z_k$, but it could equally well be thought of as
specifying functions $g_k \in \H(G_k)$ and requiring $g_k - f|_{G_k} \in
\N_k$. That is, we consider certain sequences $(w_k)$ where each $w_k$
is a \emph{coset} of $\N_k$ in $\H(G_k)$ and then we say that $f$
\term{interpolates} $(w_k)$ if, for each $k$, $f|_{G_k} \in w_k$. Simple
interpolation corresponds to the case where each $Z_k$ is a singleton
$\{ z_k \}$. Then the quotient space $\H(G_k)/\N_k$ is one dimensional
and each coset is determined by the common value of its members at
$z_k$.

Given this point of view we need to provide an appropriately normed
sequence space and define our interpolation problem. We suppress the
dependence on $p$ and $\phi$ in the notation and define the sequence
space $X_\I$ to consist of all sequences $w = (w_k)$ where $w_k \in E_K
= \H(G_k)/\N_k$ and $\| w \| = \left( \sum \| w_k \|^p \right)^{1/p} <
\infty$, where the norm of the coset $w_k$ is the quotient norm:
\begin{equation}
  \| w_k \|^p = \inf \left\{ \int_{G_k} \frac{\av{g(z) e^{-\phi(z)}}^p}{1-|z|^2}
  \,dA(z) \st g\in w_k \right\}
\end{equation}
Since every coset of $\N_k$ contains a polynomial, the norms $\| w_k \|$
are finite. It is not hard, especially in light of later results, to see
that in the case of singleton $Z_k$ this is equivalent to a space
consisting of sequences of constants $(c_k)$ satisfying
\begin{equation*}
  \| (c_k) \|^p = \sum |c_k|^p e^{-p\phi(z_k)} (1 - |z_k|^2) < \infty\,.
\end{equation*}

Now we can define the interpolation problem and interpolating sequences.
The interpolation problem is the following: given a sequence $(w_k) \in
X_\I$, find a function $f \in A^p_\phi$ such that $f|_{G_k} \in w_k$ for
every $k$. Since a coset $w_k$ can be represented by a function $g_k$ on
$G_k$ with norm arbitrarily close to that of $w_k$, we could equally
well describe the problem by: given analytic functions $g_k$ on $G_k$
for each $k$, satisfying
\begin{equation*}
  \sum_{k} \int_{G_k} \frac{\av{ g_k(z) e^{-\phi(z)}}^p}{1 - |z|^2} \,dA(z)
  < \infty
\end{equation*}
find $f \in A^p_\phi$ such that $g_k - f|_{G_k} \in \N_k$.

We say $\Z = \Union Z_k$ is an \term{interpolating sequence relative to
the scheme $\I$} if every such interpolation problem has a solution.

That is, if we define the \term{interpolation operator $\Phi$} by
letting $\Phi(f)$ be the sequence of cosets $\left(f|_{G_k} + \N_k \right)$,
then an interpolating sequence is one where $\Phi(A^p_\phi)$ contains
$X_\I$. At the moment, we do not require that $\Phi$ take $A^p_\phi$
\emph{into} $X_\I$, but we will see that it does in fact do so, and is a
bounded linear mapping.

One important step will be to show that if $\Z$ is an interpolating
sequence relative to a scheme $\I$ then the scheme must satisfy two
additional properties: (1)~there is a positive lower bound on the
distance between different $Z_k$ and (2)~there is an upper bound on the
cardinality of the $Z_k$. Schemes satisfying these two properties will
be called \term{admissible}, and our main theorem will be that a
sequence is interpolating relative to an admissible scheme if and only
if it satisfies a density inequality we will define later. An important
property of this result is that the density inequality depends only on
the sequence $\Z$ and not on the scheme itself. That is why we apply the
adjective `interpolating' to $\Z$ rather than the scheme. Also, once
this has been established, the qualification `relative to $\I$' will
become redundant.

\section{Preliminary results}

It may not be immediately obvious that $A^p_\phi$ is nontrivial. This
will follow from the following two results.

\begin{lemma}
Let $\phi$ be subharmonic and suppose there exist constants $0 < m \le
M < \infty$ such that $m \le \invL \phi(z) \le M$ for all $z\in \ID$.
If $\a > 0$ and we set
\begin{equation*}
  \tau(z) = \phi(z) - \a\log\left( \frac{1}{1-|z|^2} \right)
\end{equation*}
then
\begin{equation*}
  m - \a \le \invL \tau(z) \le M - \a\quad \text{and}\quad
    \frac{e^{-p\phi(z)}}{1 - |z|^2} = e^{-p\tau(z)}(1 - |z|^2)^{\a p - 1}
\end{equation*}
\end{lemma}

The proof is an obvious computation. Since $\tau$ satisfies the same
condition as $\phi$ if $\a$ is chosen with $\a < m$, the set of spaces
$A^p_\phi$ (ranging over all such $\phi$) are the same as set the spaces
$A^{p,\a}_\phi$ (ranging over all such $\phi$ and all $\a > 0$), whose
norms are defined by
\begin{equation*}
  \| f \|_{p,\phi,\a} = \left( \int_{\ID} \av{f(z) e^{-\phi(z)}}^p(1 -
    |z|^2)^{\a p - 1} \,dA(z) \right)^{1/p}
\end{equation*}

The following was proved in \cite{Lue04b} and also in \cite{SW13}
(stated somewhat differently and with a somewhat different proof).

\begin{lemma}\label{lem:harmonic}
Let $\phi$ be subharmonic and assume $\invL \phi$ is bounded. Then there
exists a constant $C$ and, for each $a \in \ID$, a harmonic function
$h_a$ such that the difference $\tau_a = \phi - h_a$ satisfies
\begin{enumerate}
  \item $\tau_a (z) \ge 0$ for all $z \in \ID$,
  \item $\tau_a (a) \le C\| \invL \phi \|_\infty$, and
  \item $\| (1-|z|^2)\dbar \tau_a(z) \|_\infty \le C\| \invL \phi
        \|_\infty$.\label{eq:grad}
\end{enumerate}
\end{lemma}

The last statement in the lemma was not mentioned in \cite{Lue04b}, but
comes out of the integral formula for $\phi(z) - h_0(z)$: differentiate
under the integral sign and apply standard estimates. It happens that
$C$ does not depend on $\phi$, but it is more important that it
does not depend on $a \in \ID$.

The gradient inequality \ref{eq:grad} implies the following.

\begin{lemma}\label{lem:gradient}
With the same hypotheses as Lemma~\ref{lem:harmonic} and the same $h_a$,
let $0\le R < 1$. Then $\phi(z) - h_a(z)$ is Lipschitz in the
hyperbolic metric \textup{(}with Lipschitz constant a multiple of $\| \invL \phi
\|_\infty$\textup{)}, and therefore there exists $C_R$ such that $\phi(z) -
h_a(z) \le C_R\| \invL \phi \|_\infty$ for all $z \in D(a,R)$.
\end{lemma}

Note that Lemma~\ref{lem:harmonic} allows us to write the norm of a
function in $A^{p,\a}_\phi$ as follows, where we let $H(z)$ be an
analytic function in $\ID$ with $\re H(z) = h_0(z)$
\begin{equation*}
  \int_\ID \av{ f(z) e^{-H(z)} e^{-\phi(z) + h_0(z)}}^p (1 -
  |z|^2)^{\a p - 1} \,dA(z)
\end{equation*}
The exponent $-\phi(z) + h_0(z)$ is negative, so that exponential is
bounded. Moreover, the function $(1 - |z|^2)^{\a p - 1}$ is integrable. Thus
$A^{p,\a}_\phi$ contains all bounded multiples of $\exp(H(z))$ and
so is certainly a nontrivial space.

It is easy to see that these transformations of $\phi$ (adding a
multiple of $\log(1-|z|^2)$ and subtracting the harmonic function $h_0$)
convert the original interpolation problem into an equivalent one. Thus,
it is without loss of generality that we can assmue $\phi$ already has
the properties of $\phi - h_0$ in the above lemma. Therefore, the rest
of this paper will be concerned with the following reduction of the
interpolation problem.

The function $\phi$ is \emph{positive} and subharmonic, and there exist
constants $m, M$ such that $0 < m \le \invL \phi(z) < M < \infty$ for
all $x\in\ID$. Moreover $(1 - |z|^2)\dbar \phi(z)$ is bounded. Let $\I =
\{ (G_k, Z_k), k = 1,2,3,\dots \}$ be an interpolation scheme and let $p >
0$ and $\a > 0$. For a coset $w_k \in \H(G_k)/\N_k$ define its norm
$\| w_k \|$ by
\begin{equation*}
  \| w_k \|^p = \inf\left\{ \int_{G_k} \av{ g(z) e^{-\phi(z)}}^p (1 -
  |z|^2)^{\a p - 1} \,dA(z) \st g\in w_k \right\}
\end{equation*}
Given a sequence of cosets $(w_k)$ satisfying $\sum_k \| w_k \|^p <
\infty$, the interpolation problem is to find $f \in A^{p,\a}_\phi$ such
that $f|_{G_k} \in w_k$, (i.e., \term{$f$ inerpolates $(w_k)$}). The
sequence $\Z = \Union Z_k$ is called an \term{interpolating sequence for
$A^{p,\a}_\phi$} if every such interpolation problem has a solution.

\section{Properties of interpolating sequences}\label{sec:properties}

Here we present several properties of interpolating sequences. These are
the same as the corresponding results in \cite{Lue04b} and the proofs
are, for the most part, the same. Therefore I will only indicate how a
proof differs in those cases where it does.

The first is that interpolating sequences are zero sequences. We use
$\Z(f)$ to denote the multiset (set with multiplicity) of zeros of $f$.

\begin{proposition}\label{thm:zeroset}
Given an interpolation scheme $\I$ with domains $G_k$~and clusters
$\Z_k$, if $\Z = \Union_k \Z_k$ is an interpolating sequence for
$A^{p,\a}_\phi$, then there is a function $f \in A^{p,\a}_\phi$ such that
$\Z(f) = \Z$.
\end{proposition}

The only thing we need that is different from the proof in \cite{Lue04b}
is a different reference for the fact that a subsequence of an
$A^{p,\a}_\phi$-zero sequence is also an $A^{p,\a}_\phi$-zero
sequence. This follows from \cite{Lue96}, especially section~5 where
weighted spaces of the type considered here are covered.

\begin{theorem}\label{thm:lowerbound}
Given an interpolation scheme $\I$ with clusters $\Z_k$, if $\Z =
\Union_k\Z_k$ is an interpolating sequence for $A^{p,\a}_\phi$ then
there is a lower bound $\delta > 0$ on the pseudohyperbolic distance
between different clusters of $\I$.
\end{theorem}

The proof in \cite{Lue04b} makes use of the following inewuality
\begin{equation*}
   |f'(z)(1 - |z|^2)|^p \le  \frac{C_r}{|D(z,r)|} \int_{D(z,r)} |f(w)|^p
   \,dA(w)
\end{equation*}
From this we can deduce that
\begin{equation*}
   \av{f'(z)(1 - |z|^2) e^{-\phi(z)}}^p (1-|z|^2)^{\a p - 1} \le
   \frac{C_r}{|D(z,r)|} \int_{D(z,r)} \av{f(w) e^{-\phi(w)}}^p (1-|w|^2)^{\a p - 1}
   \,dA(w)
\end{equation*}
using the inequality of Lemma~\ref{lem:gradient}. After that,
the proof is the same.

I should add that a similar inequality for $f(z)$:
\begin{equation*}
   \av{f(z) e^{-\phi(z)}}^p (1-|z|^2)^{\a p - 1} \le
   \frac{C_r}{|D(z,r)|} \int_{D(z,r)} \av{ f(w) e^{-\phi(w)} }^p
   (1-|w|^2)^{\a p - 1} \,dA(w)
\end{equation*}
shows that the unit ball of $A^{p,\a}_\phi$ is a normal family and
therefore these spaces are complete.

In \cite{Lue04b}, part of the definition of an interpolating sequence
was that the interpolation operator was bounded. We have not made that
assumption here. Thus we cannot use the open mapping principle to obtain
an \term{interpolation constant}. We nevertheless obtain one as follows.

Let $I_\Z$ consist of all functions in $A^{p,\a}_\phi$ that vanish on
$\Z$ with at least the given multiplicities. Then for any interpolation
scheme $\I$ for which $\Z$ is an interpolating sequence, there is a map
from $X_\I$ to the quotient space $A^{p,\a}_\phi/I_\Z$ taking a sequence
of cosets $(w_k)$ to the coset of functions that interpolate it. It is
straightforward to see that this mapping has closed graph and, since
both $X_\I$ and $A^{p,\a}_\phi/I_\Z$ are complete, it is continuous. If
$K$ is the norm of this mapping, then every sequence $w \in X_\I$ is
interpolated by a coset in $A^{p,\a}_\phi/I_\Z$ with quotient norm at
most $K\| w \|$. By a normal families argument, we can select a
representative function (i.e., an element of the same coset) also with
norm at most $K\| w \|$. The minimal $K$ for which this is satisfied is
called the \term{interpolation constant} for $\Z$ (relative to the
scheme $\I$).

Theorem~\ref{thm:lowerbound} implies that if $\Z$ is interpolating
relative to an interpolation scheme, then the sets $G_k$ have bounded
overlap. That is, for some constant $M$ we have $\sum \chi_{G_k}(z) \le
M$ for all $z\in \ID$. (See \cite{Lue04b} for the details.) Now, every
function $f \in A^{p,\a}_\phi$ defines a sequence of cosets $(w_k)$,
where $w_k$ is the coset determined by $f|_{G_k}$. We can therefore
estimate the norm of each coset by
\begin{equation*}
  \| w_k \|^p \le \int_{G_k} \av{f(z) e^{-\phi(z)}}^p (1 - |z|^2)^{\a p - 1}
  \,dA(z)
\end{equation*}
Summing these and using the bounded overlap, we get
\begin{equation*}
  \sum \| w_k \|^p \le M \int_{\ID} \av{f(z) e^{-\phi(z)}}^p (1 - |z|^2)^{\a p - 1}
    \,dA(z)
\end{equation*}
That is, $(w_k) \in X_\I$. Thus, for the mapping $\Phi$ taking each $f$
to its sequence of cosets we have not only $X_\I \subset
\Phi(A^{p,\a}_\phi)$, but also $\Phi(A^{p,\a}_\phi)\subset X_\I$
and $\Phi$ is bounded.

If $\I$ and $\I'$ are interpolation schemes, we will say that $\I'$ is a
\term{subscheme} of $\I$ if for each pair $(G_k',Z_k')$ of $\I'$ there
exists a pair $(G_k,Z_k)$ of $\I$ such that $G_k' = G_k$ and $Z_k'$ is a
subset (with multiplicity) of $Z_k$.

\begin{proposition}
If $\Z = \Union Z_k$ is an interpolating sequence for $A^{p,\a}_\phi$
relative to the interpolation scheme $\I = \{ (G_k, Z_k), k = 1, 2,
3,\dots \}$ and if $\I' = \{ (G_k',Z_k'), k = 1,2,3,\dots \}$ is a
subscheme, then $\Z' = \Union Z_k'$ is an interpolating sequence for
$A^{p,\a}_\phi$ relative to $\I'$. The interpolation constant for $\I'$
is less than or equal to the constant for $\I$
\end{proposition}

The proof is the same as in \cite{Lue04b}.

Invariance under M\"obius transformations is just slightly more
involved, because composition will also change the function $\phi$.
However, the new function will satisfy the same conditions as $\phi$. We
will normally want, after composition, the new weight to remain bounded
above and also satisfy a uniform lower bound at $0$. Thus, given a point
$a \in \ID$, let  $M_a(z) = (a-z)/(1-\bar a z)$, a M\"obius
transformation that maps $a$ to $0$ and is its own inverse. Given a
space $A^{p,\a}_\phi$, let $\phi_a(z) = \phi(M_a(z)) - h_a(M_a(z))$,
where $h_a$ is the harmonic function of Lemma~\ref{lem:harmonic}.

\begin{proposition}
Let $\I$ be an interpolation scheme with clusters $Z_k$ and domains
$G_k$. If $\Z=\Union Z_k$ is interpolating for $A^{p,\a}_\phi$ with
respect to $\I$ and $a\in \ID$, then $M_a(\Z)$ is interpolating for
$A^{p,\a}_{\phi_a}$ relative to the scheme $M_a(\I)$ which has
clusters $M_a(Z_k)$ and domains $M_a(G_k)$. Moreover, the
interpolation constants are the same.
\end{proposition}

\begin{proof}
The map $\Phi_a f = (fe^{-H_a}) \circ M_a (M_a')^{\a + 1/p}$ (where
$H_a$ is chosen with $\re H_a = h_a$ and say $\im H_a(a) = 0$) is an
isometry from $A^{p,\a}_\phi$ to $A^{p,\a}_{\phi_a}$. It maps the $\N_k$
associated with $Z_k$ to the $\N_k'$ associated with $Z_k'$ and
therefore maps a coset $w_k$ of $\N_k$ to a coset $w_k'$ of $\N_k'$.
Moreover, the mapping of cosets is isometric. Thus, $\Phi_a$ converts
any interpolation problem for $A^{p,\a}_{\phi_a}$ to an isometric
problem for $A^{p,\a}_{\phi}$ and the inverse converts its solution to
an isometric solution.
\end{proof}

One key requirement of an interpolating sequence is that adding a single
point to it produces an interpolating sequence (for an appropriately
augmented scheme), with a suitable estimate on the new interpolation
constant.

\begin{proposition}
Let $\I$ be an interpolation scheme with clusters $Z_k$ and domains
$G_k$. Suppose $\Z = \Union Z_k$ is an interpolating sequence for
$A^{p,\a}_\phi$ relative to $\I$ and let $z_0 \in \ID$. Suppose
there is an $\eps > 0$ such that  $\psi(z_0,Z_k) > \eps$ for every $k$.
Define a new scheme $\J$ whose domains are all the domains of $\I$ plus
the domain $G_0 = D(z_0, 1/2)$ and whose clusters $W_k$ are all the
$Z_k$ plus $W_0 = \{ z_0 \}$. Then
$\W = \{ z_0 \} \union \Z$ is an interpolating sequence relative
to the scheme $\J$.

If $K$ is the interpolation constant for $\Z$ then the constant for $\W$
is at most $C K/\eps$, where $C$ is a positive constant that depends
only on the space $A^{p,\a}_\phi$.
\end{proposition}

\begin{proof}
Some of the proof in \cite{Lue04b} is simplified by the symmetry of the
weights, so we will have to add a little detail. Without loss of
generality we may assume $z_0 = 0$. Suppose we wish to interpolate a
sequence of cosets $w = (w_k, k = 0,1,2,\dots)$ in $X_\J$, with $\| w \|
= 1$. Choose representative functions $g_k$ of minimal norm for all
$w_k$ with $k \ne 0$. For $k=0$, $w_0$ contains a constant function $g_0$.
This may not be the minimizing representative, but from the inequality
\begin{equation*}
  \av{ f(0) e^{-\phi(0)} }^p \le C \int_{G_0} \av{ f(z) e^{-\phi(z)} }^p
    (1 - |z|^2)^{\a p - 1} \,dA(z)
\end{equation*}
we can estimate $\| w_0 \|$ within a constant factor by using
$g_0$.

Now consider the functions $f_k = (g_k - g_0)/z$ for $k\ne 0$. One
easily estimates
\begin{equation*}
  \int_{G_k} \av{f_k(z) e^{-\phi(z)}}^p (1 - |z|^2)^{\a p - 1} \,dA(z) \le
    \frac{C_p}{\eps^p} \left( \| w_k \|^p + | g_0 |^p \mu(G_k) \right)
\end{equation*}
Where $\mu$ is the measure $e^{-p\phi(z)}(1 - |z|^2)^{\a p - 1} \,dA(z)$.
Therefore the sequence of cosets $(u_k)$ represented by $(f_k)$ belongs
to $X_\I$, having norm at most $C_p^{1/p}(1 + C\mu(\ID))/\eps$ for some
$C$.

Since $\Z$ is interpolating, there exist $f \in A^{p,\a}_\phi$ that
interpolates $(u_k)$ with norm at most $K\| (u_k) \| = CK/\eps$ for some
constant. Then $zf(z) + g_0$ interpolates $(w_k)$ with norm at most
$CK/\eps$ for some other constant $C$.
\end{proof}

We say that $\Z$ has bounded density if for $0<R<1$ there is a finite
constant $N=N_R$ such that every disk $D(a,R)$, $a\in \ID$, contains no
more than $N$ points (counting multiplicity). If there is a finite upper
bound for some $R \in (0,1)$ then there is a finite upper bound for any
$R\in (0,1)$, although the bounds will be different. We will show that
an interpolating sequence relative to a scheme $\I$ must have bounded
density. Given the bounded overlap of the domains and the uniform
separation between clusters, it is enough to show that there is an upper
bound on the number of points in each cluster (counting multiplicity).

\begin{theorem}
If $\Z$ is an interpolating sequence for $A^{p,\a}_{\phi}$ relative
to an interpolation scheme $\I$ then there is a finite upper bound $B$ on the
number of points, counting multiplicity, in each cluster $Z_k$ of $\I$.
\end{theorem}

The proof is the same as in \cite{Lue04b} except we use M\"obius
transformations $M_a$ to map $A^{p,\a}_{\phi}$ to $A^{p,\a}_{\phi_a}$ as
before. It is important that there is a lower bound on $\phi_a(0)$
independent of $a$. This means there is also a lower bound on $\phi_a$
on compact sets, allowing the normal families argument to proceed.

As in \cite{Lue04b}, we now have two additional conditions that the
scheme $\I$ must satisfy in order for the sequence $\Z = \Union Z_k$ to
be interpolating, and we call such schemes admissible.

Summarizing, we have defined $\I = \{ (G_k,Z_k), k=1,2,3,\dots \}$ to be
an interpolation scheme if it satisfies properties P1 and P2 below. We
will say $\I$ is an \term{admissible interpolation scheme} if it also
satisfies P3 and P4:

\begin{itemize}
\item[(P1)] There is an $R < 1$ such that the pseudohyperbolic diameter
     of each $G_k$ is at most $R$.
\item[(P2)] There is an $\epsilon > 0$ such that $(\Z_k)_\epsilon
    \subset G_k$ for every $k$.
\item[(P3)] There is a $\delta > 0$ such that for all $j\ne k$ the
    pseudohyperbolic distance from $\Z_j$ to $\Z_k$ is at least
    $\delta$.
\item[(P4)] There is an upper bound $B$ on the number of points
    (counting multiplicity) in each cluster $\Z_k$
\end{itemize}

As in \cite{Lue04b}, any sequence $\Z$ with bounded density can be
subdivided into clusters $Z_k$, with associated open sets $G_k$, so that
the result is an admissible interpolation scheme. It will not be needed,
but it may be interesting that the scheme produced satisfies $G_k =
(Z_k)_\eps$ for some $\eps>0$, and moreover the $G_k$ are disjoint. One
could therefore `fill in the holes' and have a scheme with simply
connected domains.

\section{Zero sets, density, and the
\texorpdfstring{$\dbar$}{dbar}-problem}\label{sec:density}

The following perturbation result differs little in proof from the
version in \cite{Lue04b}. The phrase \term{interpolation invariants}
means quantities, such as the interpolation constant, that are unchanged
under a M\"obius transformation of the disk. This includes the numbers
$p$ and $\alpha$ and in this paper also the estimates on $\invL \phi$.

\begin{proposition}\label{thm:stability}
  Let $\I$ be an admissible interpolation scheme with domains $G_k$
  and clusters $\Z_k$. Assume $\Z = \Union_k \Z_k$ is an
  interpolating sequence for $A^{p,\alpha}_\phi$ with interpolation
  constant $K$. For each $k$ let $\b_k$ be defined by $\b_k(z) = r_k
  z$ and let $\J$ be the interpolation scheme with domains $D_k =
  \b_k(G_k)$ and clusters $W_k = \b_k(Z_k)$. Let $\W = \Union_k
  W_k$.

  There exists an $\eta > 0$ depending only on interpolation invariants
  such that if $\psi(\b_k(z),z) < \eta$ for all $z\in G_k$ and for
  all $k$, then $\W$ is an interpolating sequence for
  $A^{p,\alpha}_\phi$ relative to $\J$. Its interpolation constant can
  be estimated in terms of $\eta$ and interpolation invariants of $\I$.
\end{proposition}

One stage in the proof in \cite{Lue04b} is an estimate of $|f(z/r_k) -
f(z)|^p$ by a small multiple of the average of $|f|^p$ on $(D_k)_{1/2}$
(the new domains expanded by pseudohyperbolic distance $1/2$).
This particular step can be done similarly when weighting with
$e^{-p\phi}(1 - |z|^2)^{\a p - 1}$. This relies mostly on the fact that $\phi$ is
Lipschitz. The rest of the proof is essentially the same.

In \cite{Lue96} it was shown that the following function could be used
to determine whether a sequence $\Z$ in $\ID$ is a zero sequence for a
variety of analytic function spaces:
\begin{equation*}
  k_\Z(\z) = \sum_{a\in \Z} k_a(z) = \sum_{a\in \Z}
    \frac{(1-|a|^2)^2}{|1-\bar az|^2}\frac{|z|^2}{2}
\end{equation*}
where a point with multiplicity $m$ occurs $m$ times in the sum. In
particular, $\Z$ is a zero set if and only if a certain weighted
function space is nontrivial. In our current context (covered in the
last section of \cite{Lue96}), we have the following theorem.

\begin{theorem}
Let $\Z$ be a sequence in $\ID$. Define the function $k_{\Z}$ as above.
The following are equivalent.
\begin{enumerate}
\item $\Z$ is a zero set for some function in $A^{p,\a}_\phi$.
\item There exists a \textbf{nowhere zero} analytic function $F$
    such that
\begin{equation}\label{eq:finite}
  \int_{\ID} \av{F(\z)e^{-\phi(\z)}}^p e^{pk_\Z(\z)} (1-|\z|^2)^{\a p-1}
        \,dA(\z)< \infty
\end{equation}

\item There exists a \textbf{nonzero} analytic function $F$
    satisfying \eqref{eq:finite}.
\end{enumerate}
\end{theorem}

The integral in \eqref{eq:finite} defines a norm that determines a space we
will call $A^{p,\a}_{\phi,\Z}$. Then $\Z$ is a zero set for
$A^{p,\a}_{\phi}$ if and only if $A^{p,\a}_{\phi,\Z}$ is non trivial.

Moreover, if we define
\begin{equation*}
  \Psi_\Z (\z) = z^m\prod_{\substack{a\in \Z\\ a \ne 0}} \bar a\frac{a-z}{1-\bar az}
    \exp\left( 1 - \bar a\frac{a-z}{1-\bar az} \right)
\end{equation*}
(where $m$ is the multiplicity of the origin, if it belongs to $\Z$, and
zero otherwise) then $f \mapsto f/\Psi_\Z$ is a one-to-one
correspondence between functions in $A^{p,\a}_{\phi}$ that vanish on
$\Z$ to at least the given multiplicities and $A^{p,\a}_{\phi,\Z}$. We
will be applying this only when $\Z$ has no points in $D(0,\delta)$ for
some fixed $\delta > 0$, in which case the value of $|\Psi_\Z(0)| =
\prod |a|^2e^{1-|a|^2}$ can be estimated from below in terms of $\delta$
and the density of $\Z$.

Note that the convergence of the product defning $\Psi_\Z$ requires the
sequence $(1 - |a|^2)^2$, $a\in \Z$, to be summable. This follows from
the formula (1) in \cite{Lue96} in light of the discussion in section~5
of that paper. For interpolating sequences, which have bounded density,
this is automatically true without any need for the results in
\cite{Lue96}.

In the case where $\phi \equiv 0$, the paper \cite{Lue04b} showed that
$\Z$ is an interpolating sequence if and only a certain density
condition is satisfied. In the general case, that density condition will
involve integrals of $\phi$. It was also shown that this is equivalent to
bounds on the solutions $u$ of the $\dbar$-equation
\begin{equation*}
  (1 - |z|^2)\dbar u = f
\end{equation*}
in a certain weighted function space. In the general case let
$L^{p,\a}_{\phi,\Z}$ be the measurable function version of
$A^{p,\a}_{\phi,\Z}$. We need a bounded operator on this space that maps
$f$ to a solution $u$.

\begin{theorem}\label{thm:main}
Let $\Z$ be a set with multiplicity in $\ID$, $p \ge 1$, $\a>0$, and
$\phi$ a positive subharmonic function satisfying $0 < m \le \invL
\phi < M < \infty$ in $\ID$. The following are equivalent:
\begin{enumerate}
  \item $\Z$ is an interpolating sequence for $A^{p,\a}_\phi$ relative
        to any admissible interpolation scheme.\label{IS1}
  \item $\Z$ is an interpolating sequence for $A^{p,\a}_\phi$
        relative to some interpolation scheme.\label{IS2}
  \item The upper uniform density $S_\phi^+(\Z)$ \textup{(}defined
        below\textup{)} is less than $\a$.\label{UUD}
  \item $\Z$ has bounded density and the $\dbar$-problem has a bounded
        solution operator on $L^{p,\a}_{\phi,\Z}$.\label{DBAR}
\end{enumerate}
\end{theorem}

We postpone the proof to discuss the density condition. We prefer to use
the following summation to define density. It was shown in \cite{Lue04a} to be
equivalent to the usual one for the standard weights.

For $r \in (0,1)$ let
\begin{equation*}
    \hat k_\Z(r) = \frac{1}{2\pi}\int_0^{2\pi} k_\Z(re^{it}) \,dt =
    \frac{r^2}{2} \sum_{a\in \Z} \frac{(1 - |a|^2)^2}{1 -
    |a|^2r^2}
\end{equation*}
then let
\begin{equation*}
  S(\Z, r) = \frac{\hat k_{\Z}(r)} {\log\left( \frac{1}{1 - r^2} \right)}
\end{equation*}

For each $a\in \ID$, let $\Z_a = M_a(\Z)$, where as before $M_a$ is the
M\"obius transformation exchanging $a$ and $0$.
In case $\phi \equiv 0$, the density we used in \cite{Lue04b} was
$S^+(\Z)$, defined by
\begin{equation*}
  S^+(\Z) = \limsup_{r\to 1-} \,\sup_{a\in\ID} S(\Z_a,r)
\end{equation*}
It was shown in \cite{Lue04a} that this is equivalent to the usual
\term{upper uniform density} $D^+$ for sets $\Z$ (as defined in
\cite{Sei04} for example). The density inequality equivalent to
interpolation in $A^{p,\a}$ (where $\phi \equiv 0$) is that $S^+(\Z) <
\a$. (In \cite{Lue04a} and \cite{Lue04b}, the condition was written as
$S^+(\Z) < (\a + 1)/p$, but the number $\a$ there was the exponent of
$(1-|z|^2)$ that we are writing here as $\a p - 1$.)

For the more general $\phi$, our density condition has to incorporate
$\phi$. Let
\begin{equation*}
  \hat \phi (r) = \frac{1}{2\pi}\int_0^{2\pi} \phi (re^{it}) \,dt
    - \phi(0)
\end{equation*}
and define
\begin{equation*}
  S_\phi (\Z, r) = \frac{\hat k_{\Z}(r) - \hat \phi(r)}
    {\log\left( \frac{1}{1 - r^2} \right)}
\end{equation*}
and finally
\begin{equation*}
  S_\phi^+ (\Z) = \limsup_{r\to 1-} \,\sup_{a\in\ID} S_{\phi_a}(\Z_a,r)
\end{equation*}

As part of the proof, we need to be able to express this density in
terms of the invariant Laplacian of the functions involved. This follows
easily from the following, obtained from Green's formula. Recall that
$d\lambda(z)$ is the invariant messure $dA(z)/(1 - |z|^2)^2$:
\begin{equation*}
  \hat\phi(r) = \frac{1}{\pi} \int_{r\ID} \invL\phi(z) \log \left(
    \frac{r^2}{|z|^2} \right) \,d\lambda(z)
\end{equation*}
A similar formula holds for $\hat k_\Z$. If we combine these two
formulas, plus one for $\log\left( \frac{1}{1-|z|^2} \right)$ we get the
following fromula
\begin{equation}\label{eq:IN}
  S_\phi (\Z,r) - \a = \frac{1}{\pi\log \left( \frac{1}{1 - r^2} \right)}
  \int_{r\ID} \invL \left ( k_\Z(z) - \phi(z) - \a \log\left( \frac{1}{1
  - |z|^2} \right) \right) \log\frac{r^2}{|z|^2} \,d\lambda(z)
\end{equation}
This relies on the calculation
\begin{equation*}
  \frac{1}{\pi} \int_{r\ID} \log\frac{r^2}{|z|^2} \,d\lambda(z) =
  \log\left( \frac{1}{1 - r^2} \right)\,.
\end{equation*}
If we temporarily let
\begin{align*}
  \tau(\z)   &=  k_\Z(\z) - \phi(\z) - \a \log\left( \frac{1}{1 - |\z|^2}
                \right)\\
  \sigma_r(\z) &= \frac{ \log\frac{r^2}{|\z|^2} \chi_{r\ID}(\z)}
                {\pi\log \left( \frac{1}{1 - r^2} \right)}
\end{align*}
then invariant nature of the formula in \eqref{eq:IN} allows us to write
\begin{equation}\label{eq:density_at_a}
  S_{\phi_a} (\Z_a,r) - \a = \frac{1}{\pi\log \left( \frac{1}{1 - r^2} \right)}
  \int_{D(a,r)} \invL \tau(z) \log\frac{r^2}{|M_a(z)|^2} \,d\lambda(z)
\end{equation}
and then the right side of equation~\eqref{eq:density_at_a} is the
\term{invariant convolution} of $\invL\tau$ and $\sigma_r$. That is
\begin{equation*}
  S_{\phi_a} (\Z_a,r) - \a = (\invL\tau) * \sigma_r (a) \equiv \int_\ID
  (\invL\tau(z)) \sigma_r(M_a(z)) \,d\lambda(z).
\end{equation*}
We know that the invariant convolution has the following
properties if one of the functions is radially
symmetric (as is $\sigma_r $):
\begin{align*}
   \tau * \sigma_r &= \sigma_r * \tau\\
   (\invL\tau)*\sigma_r &= \invL(\tau*\sigma_r)
\end{align*}
Therefore, the density condition~\ref{UUD} of Theorem~\ref{thm:main} is
equivalent to the requirement that there exists an $r_0 \in (0,1)$ and
an $\eps > 0$ such that the invariant Laplacian $\invL (\tau*\sigma_r)$
is bounded above by $-\eps$ for all $r > r_0$. We note that this means
we can (and will) invoke Lemma~\ref{lem:harmonic} on $-(\tau*\sigma_r)$.
Note also that the fact that $\tau$ is Lipschitz in the hyperbolic
metric shows that  $\tau - \tau * \sigma_r$ is a bounded function with a
bound that depends on $r$.

Recall that originally the space $A^p_\phi$ had $\a=0$ and no
requirement that $\phi$ be positive. We modified it by subtracting $\a
\log\left(\frac{1}{1-|z|^2}\right)$ and a harmonic function.
Consequently, the combination $\phi(z) + \a\log\left( \frac{1}{1-|z|^2}
\right)$ that appears in equation~\eqref{eq:density_at_a} is in fact the
original exponent defining $A^p_\phi$, up to an added harmonic function.
Therefore the means and invariant Laplacian of $\phi(z) + \a\log\left(
\frac{1}{1-|z|^2} \right)$ are the same as those of the original $\phi$.

\section{Proofs}

The proof of Theorem~\ref{thm:main} proceeds just as in \cite{Lue04b},
so we will only describe the highlights.

Given an interpolating sequence $\Z$ for an admissible scheme $\I$, we
can delete the pairs $(G_k,Z_k)$ where $Z_k$ meets $D(0,1/2)$ and add
the domain $G_0 = D(0,1/2)$ with cluster $Z_0 = \{ 0 \}$ to obtain a new
scheme $\J$. Then a function $f$ exists with $f(0) = 1$ that vanishes on
the union $\Z'$ of the remaining clusters. We get an estimate on the
$A^{p,\a}_{\phi}$-norm of $f$ that depends only on the data about $\I$
that are invariant under M\"obius thransfomations of $\I$. We can
normalize $f$ and then we get a lower bound on the value of $f(0)$. We
can modify $f$ so that it vanishes only on $\Z'$, still having norm $1$
and retaining a lower bound on $f(0)$.

We then divide $f$ by $\Psi_{\Z'}$ to get a nonvanishing function in
$A^{p,\a}_{\phi,\Z'}$. Since $\Z$ and $\Z'$ differ only in a finite
number of points (the number of which can be estimated in terms of
interpolation invariants), this space is equivalent to
$A^{p,\a}_{\phi,\Z}$. We can do all of this after first perturbing $\I$
inward an amount small enough that the perturbed sequence $\W$ remains
an interpolating sequence and so we obtain $f\in A^{p,\a}_{\phi,\W}$
which we normalize to have norm $1$ and we still obtain a lower bound on
$f(0)$.

Following \cite{Lue04b}, we can perturb $\W$ back outward to $\Z$ and
obtain a constant $\b < 1$ and a new function $g$ that satisfies
\begin{equation*}
  \int_{\ID} \av{g(z) e^{k_\Z(z)}}^{p/\b} e^{-p\phi(z)} (1 -
    |z|^2)^{\a p - 1} \,dA = 1
\end{equation*}
while retaining a lower bound on $g(0)$. Solve an extremal problem:
maximize $|g(0)|$ subject to the above equality to obtain a new
function $g$ such that the above integrand defines a Carleson measure,
from which we obtain a constant $C$ such that
\begin{equation*}
   \av{g(z) e^{k_\Z(z)}}^{p/\b} e^{-p\phi(z)} (1 -
   |z|^2)^{\a p} \le C \quad \text{for all $z\in \ID$.}
\end{equation*}

Now consider
\begin{multline*}
  \frac{1}{2\pi}\int_0^{2\pi} \frac{p}{\b} \log |g(re^{it})| +
        \frac{p}{\b}k_\Z(re^{it}) - p(\phi(re^{it})-\phi(0)) + \a p
        \log (1 - |r|^2) \,dt \\
  \begin{aligned}
    &\le \log \left( \frac{1}{2\pi} \int_0^{2\pi} \av{g(re^{it})
        e^{k_\Z(re^{it})}}^{p/\b}e^{-(\phi(re^{it})-\phi(0))}(1 -
        |r|^2)^{\a p} \,dt \right) \\
    &\le \log C,
  \end{aligned}
\end{multline*}
The extra factor $e^{\phi(0)}$ can be included because we have an
estimate on $\phi(0)$ in terms of $\| \invL \phi \|_\infty$. We multipy
this by $\b/p$ and use the fact that the mean of $\log|g|$ exceeds
its value at $0$ to get
\begin{equation*}
  \hat k_\Z(r) - \b \hat \phi - \b \a \log\frac{1}{1 - r^2}
    \le C - \log |g(0)|
\end{equation*}
We can rewrite this in terms of the invariant Laplacian as discussed
previously (and incorporate $\log |g(0)|$ into the constant):
\begin{equation}\label{eq:IL}
   \int_{r\ID} \invL \left( k_\Z(r) - \b\phi - \b \a
   \log\frac{1}{1 - |z|^2}  \right)
   \log \left( \frac{r^2}{|z|^2} \right) \,d\lambda \le C
\end{equation}
We can estimate as follows: for some $\eps > 0$
\begin{equation*}
   \b \invL \left( \phi + \a \log\left( \frac{1}{1 - |z|^2}  \right)  \right)
    \le \invL \left( \phi + (\a - 2\eps) \log\frac{1}{1 - |z|^2}  \right)
\end{equation*}
because the invariant Laplacian on the left side is bounded away from
$0$ and the invariant Laplacian of the $\log$ expression is constant.
Inserting this into \eqref{eq:IL} and then rewiting the result in terms
of means, we obtain
\begin{equation*}
  \hat k_\Z(r) - \hat \phi \le (\a - 2\eps) \log \frac{1}{1 - r^2} + C
\end{equation*}
Divide this by $\log\frac{1}{1 - r^2}$ and then, for $r$ sufficiently
near $1$ we have
\begin{equation}\label{eq:density_at_0}
  \frac{\hat k_\Z(r) - \hat \phi}{\log\frac{1}{1 - r^2}} \le \a - \eps
\end{equation}
Since the constants have estimates that are uniform over all M\"obius
transforms, we can replace $\Z$ by its M\"obius transforms $\Z_a$ and
take the supremum of the above inequality over all $a$ to obtain the
required density condition~\ref{UUD}:
\begin{equation*}
  \sup_{a\in \ID} \frac{\hat k_{\Z_a}(r) - \hat \phi_a}{\log
    \frac{1}{1-r^2}} \le \a - \eps
\end{equation*}
for all $r$ sufficiently close to $1$.

As we saw at the end of section~\ref{sec:density}, the
condition~\eqref{eq:density_at_0} is equivalent to the existence of an a
negative upper bound on the invariant Laplacian of the convolution
$\tau*\sigma_r$ where
\begin{align*}
  \tau(\z)   &=  k_\Z(\z) - \phi(\z) - \a \log\left( \frac{1}{1 - |\z|^2}
                \right)\\
  \sigma_r(\z) &= \frac{ \log\frac{r^2}{|\z|^2} \chi_{r\ID}(\z)}
                {\pi\log \left( \frac{1}{1 - r^2} \right)}
\end{align*}
Then Lemma~\ref{lem:harmonic} (applied to $-\tau*\sigma_r(\z)$) provides
us with a harmonic function $h$ such that $\tau*\sigma_r(\z) + h(\z)$ is
everywhere negative and there is a lower bound on its value at $0$ in
terms of the sup norm of the invariant Laplacian. Since $\tau -
\tau*\sigma_r$ is bounded, we get a similar result for $\tau$ itself.
That is, there exists constants $C$ and $\eps$ (depending only on
$\phi$, $p$, $r$ and the scheme $\I$) and a harmonic function $h$ such
that
\begin{equation*}
  k_\Z(\z) - \phi(\z) + h(\z) \le (\a-\eps) \log\left( \frac{1}{1 -
    |\z|^2} \right)\\
\end{equation*}
and
\begin{equation*}
  k_\Z(0) - \phi(0) + h(0) \ge -C
\end{equation*}
Using the uniformity of our estimates over M\"obius transformations, we
obtain for each $a \in \ID$ a harmonic function $h_a$ such that
\begin{equation*}
  k_\Z(\z) - \phi(\z) + h_a(\z) \le (\a-\eps) \log\left( \frac{1}{1 -
    |M_a(\z)|^2} \right)\\
\end{equation*}
and
\begin{equation*}
  k_\Z(a) - \phi(a) + h_a(a) \ge -C
\end{equation*}
Exponentiating, we get holomorphic functions $g_a(z)$ and constants
$\delta > 0$ and $C$ such that
\begin{equation*}
  \av{g_a(\z) e^{k_\Z(\z) - \phi(\z)}}  \le \frac{1}{\left( 1 -
    |M_a(\z)|^2 \right)^{\a-\eps}} \\
\end{equation*}
and
\begin{equation*}
   \av{g_a(a) e^{k_\Z - \phi}} \ge \delta
\end{equation*}
These functions allow us to construct a solution of the $\dbar$-equation
exactly as in \cite{Lue04a}. That is the solution of $(1 - |z|^2)\dbar
u(z) = f(z)$ is given by
\begin{equation*}
  u(z) = \frac{1}{\pi} \sum_{j=1}^\infty g_{a_j}(z) \int_{\ID}
    \frac {\g_j(w) f(w)} {g_{a_j}(w)} \frac {(1 - |w|^2)^{m-1}}
    {(z - w) ( 1 - \bar wz )^m} \,dA(w)
\end{equation*}
where $\g_j$ is a suitable partition of unity and $m$ is a sufficiently
large integer. The lower estimate on $g_{a_j}e^{k_\Z - \phi}$ at $a_j$
allows us to divide by it on the support of $\g_j$, provided that
support is sufficiently small. The upper estimates allow us to show that
the operator is bounded on $L^{p,\a}_{\phi,\Z}$. This shows that
condition \ref{UUD} of theorem~\ref{thm:main} implies condition
\ref{DBAR}.

Finally, given solutions with bounds for the $\dbar$-equation, we can
solve any interpolation problem just as in \cite{Lue04b}. This ends the
(sketch of the) proof.

If one returns to the original space $A^p_\phi$, the theorem can be
restated as follows:

\begin{theorem}
Let $\Z$ be a set with multiplicity in $\ID$, $p \ge 1$, and
$\phi$ a subharmonic function satisfying $0 < m \le \invL
\phi < M < \infty$ in $\ID$. The following are equivalent:
\begin{enumerate}
  \item $\Z$ is an interpolating sequence for $A^{p}_\phi$ relative
        to any admissible interpolation scheme.
  \item $\Z$ is an interpolating sequence for $A^{p}_\phi$
        relative to some interpolation scheme.
  \item $S_\phi^+(\Z) < 0$.
  \item $\Z$ has bounded density and the $\dbar$-problem has a bounded
        solution operator on $L^{p}_{\phi,\Z}$.
\end{enumerate}
\end{theorem}

\section{\texorpdfstring{$p$}{p} less than \texorpdfstring{$1$}{1}}

Most of the considerations that went into the proof of
theorem~\ref{thm:main} apply equally well to all $p \in (0,\infty)$.
However the last step, constructing a solution of the $\dbar$-equation,
fails when $p < 1$: the integrals in question may not exist when $f$ is
not locally integrable. The way around this deficiency is to replace the
domain of the $\dbar$-equation (normally $L^{p,\a}_{\phi,\Z}$) with a
smaller one. One example: all measurable functions $f$ that are locally
in $L^q$ for some $q \in [1,\infty]$ and such that $m_q(f) \in
L^{p,\a}_{\phi,\Z}$ where
\begin{equation*}
  m_q(f)(\z)  =
  \begin{cases}
    \frac{1}{|D(\z,1/2)|} \int_{D(\z,1/2)} |f|^q  \,dA & q < \infty\\
    \sup_{z\in D(\z,1/2)}  |f(w)|                      & q = \infty
   \end{cases}
\end{equation*}
All holomorphic functions wind up in this space, even with $q = \infty$.
Moreover, when proving \ref{DBAR}${}\Rightarrow{}$\ref{IS1} of
theorem~\ref{thm:main}, the function to which one applies the solution
operator belongs to this space (even with $q=\infty$). The proof in
\cite{Lue04a} of the boundedness of this solution works here for
$p<1$ just as well as for $p \ge 1$.

Therefore, Theorem~\ref{thm:main} is valid for $p < 1$ provided only
that in part~\ref{DBAR} we replace the space $L^{p,\a}_{\phi,\Z}$ with
this modified version.

\section{Application to O-interpolation}

Let $\Z$ be a sequence of distinct points in $\ID$ having bounded
density, and let $c_a$, $a \in \Z$, be sequence of values satisfying
\begin{equation}\label{eq:finiteness}
  \sum_{a \in \Z} |c_a|^p \frac{e^{-p\phi(a)}}
  {\delta_a^{pn_a}}(1-|a|^2) < \infty
\end{equation}
where $\delta_a$ is the pseudohyperbolic distance from $a$ to the
nearest point in $\Z\setminus\{ a \}$ and $n_a$ is the number of points
of $\Z$ in $D(a,1/2)$. Then O-interpolation consists of finding a
function $f \in A^p_\phi$ satisfying $f(a) = c_a$ for all $a\in \Z$.

Just as in the addendum to \cite{Lue04b} (the last section), we can
provide an admissible scheme $\I = \{ (G_k,Z_k), k = 1,2,3,\dots \}$ for
$\Z$ and define functions $f_k$ on $G_k$  that have the values $c_a$ at
the points $a$ of $\Z$ that lie in $G_k$. Moreover, the $L^p$-norms of these
functions provides an upper bound for the norm $\| w_k \|$ of the cosets
determined by $f_k$ and these are shown to be less than
\begin{equation*}
  C \sum_{a \in Z_k} |c_a|^p \frac{e^{-p\phi(a)}} {\delta_a^{pn_a}}(1-|a|^2)
\end{equation*}
with $C$ independent of $k$. Thus the finiteness
condition~\eqref{eq:finiteness} dominates $\sum \| w_k \|^p$.

Thus we have created an interpolation problem relative to the scheme $\I$
whose solution would be a function $f$ satisfying $f(a) = c_a$. The density
condition now implies that a solution exists in $A^p_{\phi}$. That is,
the density condition implies O-interpolation.


\newcommand{\noopsort}[1]{} \providecommand{\bysame}{\leavevmode\hbox
to3em{\hrulefill}\thinspace}
\providecommand{\MR}{\relax\ifhmode\unskip\space\fi MR }
\providecommand{\MRhref}[2]{%
  \href{http://www.ams.org/mathscinet-getitem?mr=#1}{#2}
}
\providecommand{\href}[2]{#2}

\end{document}